\documentclass[12]{article}

\usepackage[margin=1in]{geometry}
\usepackage{authblk}
\usepackage{amsmath,amssymb,amsfonts,mathtools,mathrsfs,amsthm}  
\usepackage{thmtools, thm-restate}
\usepackage{algorithm,algpseudocode}
\usepackage{caption}
\usepackage[round]{natbib} 
\allowdisplaybreaks

\newcommand{\bs}[1]{\boldsymbol{#1}}
\newcommand{\x}{\boldsymbol{x}}

\declaretheorem{definition}
\declaretheorem{theorem}
\declaretheorem{lemma}
\declaretheorem{proposition}

\declaretheorem{assumption}
\declaretheorem{remark}

\begin{document}

\title{A Concentration Result of Estimating Phi-Divergence using Data Dependent Partition} 
\author[1]{Fengqiao Luo}
\author[2]{Sanjay Mehrotra}
\affil[1,2]{Department of Industrial Engineering and Management Science, Northwestern University}

\maketitle

\begin{abstract}
Estimation of the $\phi$-divergence between two unknown probability distributions using empirical data is a fundamental problem in 
information theory and statistical learning. We consider a multi-variate generalization of the data dependent partitioning method
for estimating divergence between the two unknown distributions. 
Under the assumption that the distribution satisfies a power law of decay, we provide a convergence rate result for this method on the 
number of samples and hyper-rectangles required to ensure the estimation error is bounded by a given level with a given probability. 
\end{abstract}

\section{Introduction} 
Let $P$ and $Q$ be two probability distributions on defined on $(\mathbb{R}^d, \mathcal{B}_{\mathbb{R}^d})$, 
where $\mathcal{B}_{\mathbb{R}^d}$ is the Borel measure on $\mathbb{R}^d$.
The $\phi$-divergence of $Q$ from $P$ is defined as:
\begin{equation}
D_{\phi}(P||Q)=\int_{\mathbb{R}^d}\phi\left(\frac{dP}{dQ} \right)dQ.
\end{equation}
The $\phi$-divergence family includes the Kullback-Leibler (KL) divergence \citep{KL1951}, 
Hellinger distance \citep{nikulin2001}, total variation distance, $\chi^2$-divergence,
$\alpha$-divergence among others. Many other information-theoretic quantities such as entropy and mutual 
information can be formulated as special cases of $\phi$-divergence.
When the distributions $P$ and $Q$ are unknown, the estimate of $D_{\phi}(P||Q)$ based on the i.i.d. samples from $P$ and $Q$ 
is a fundamental problem in information theory and statistics.  
The estimation of $\phi$-divergence as well as entropy and mutual information has many important applications.

In statistics, divergence estimators can be used for hypothesis testing of whether two sets of i.i.d. samples are drawn from the same distribution. Methods
for divergence estimation can also provide sample sizes required to achieve given significance level in hypothesis testing \citep{2004-hero_multivar-f-diverg-confid}. 
Divergence is also applicable as a loss function in evaluating and optimizing the performance of density estimation methods \citep{hall1987_KL-loss-density-est,2009-wang_diverg-est-multidim-dens-knn}. Similarly, entropy estimators can be used to build the goodness-of-fit tests for the entropy of a random vector \citep{goria2005_rand-vec-entr-estor}.  Entropy estimation is applicable for parameter estimation in semi-parametric regression models, where the distribution function of the error is unknown \citep{wolsztynski2005_min-entr-est-semi-param-mod}.

In machine learning, many important algorithms for regression, classification, clustering, etc operate on the finite vector space.  
Using divergence estimation can potentially expand the scope of many machine learning algorithms by enabling them to 
operate on space of probability distributions. For example, divergence can be employed to construct kernel functions defined 
in the probability density function space. The kernel basis constructed in this way leads to better performance in multimedia classification 
than using the ordinary kernel basis defined in the sample space \citep{moreno2004_KL-diverg-svm-multimedia-appl}.
Using statistical divergence to construct kernel functions, and i.i.d. samples to estimate the divergence are also applied for image
classification \citep{poczos2012_nonparam-kernel-estor-image-classif}.

Entropy estimation and mutual information estimation have been used for texture classification 
\citep{hero2002_appl-entropic-span-graph,hero2002_alpha-diverg-class-ind-retr}, 
feature selection \citep{2005-peng_feat-select-mutual-inf},
clustering \citep{2007-aghagolzadeh_hierach-cluster-mutual-inf},
optimal experimental design \citep{2007-lewi_real-time-adpt-inf-theor-opt-neuro}
fMRI data processing \citep{2009-chai_expl-func-connect-brain-inf-analy},
prediction of protein structures \citep{adami2004_inf-thy-molecular-bio},
boosting and facial expression recognition \citep{2005-shan_cond-mutual-inf-boosting-facial-exp-recog},
independent component and subspace analysis \citep{2003-miller_ICA-spacings-est-entropy,2007-szabo_undercomp-blind-subspace-deconvolution},
as well as for image registration \citep{hero2002_appl-entropic-span-graph,hero2002_alpha-diverg-class-ind-retr,2006-kybic_incremt-update-nn-high-dim-entropy-est}.

We now put our work in the context of concentration results known in the literature.
\citet{liu2012_exp-concent-ineq-mutual-inf-est} derived exponential-concentration bound for an estimator of the two-dimensional
Shannon entropy over $[0,1]^2$. 
\citet{singh2014_gen-exp-concent-ineq-renyi-diverg-estor} generalized the method in \citep{liu2012_exp-concent-ineq-mutual-inf-est}
to develop an estimator of mutual R\'{e}nyi divergence for a smooth H\"{o}lder class of densities on 
$d$-dimensional unit cube $[0,1]^d$ using kernel functions, and they also derived an exponential concentration inequality 
on the convergence rate of the estimator. 
\citet{perez-cruz-2008_est-inf-thy-meas-cont-rv} constructed an estimator of KL divergence and differential entropy, 
using $k$-nearest-neighbor approach to approximate the densities at each sample point. They show that the  
estimator converges to the true value almost surely.
\citet{pal2010_est-renyi-entrop-mutual-inf-gen-nng} constructed estimators of R\'{e}nyi entropy and mutual information 
based on a generalized nearest-neighbor graph. They show the almost sure convergence of the estimator, and
provide an upper bound on the rate of convergence for the case that the density function is Lipschitz continuous. 

\citet{wang2005_diverg-est-cont-distr-data-dept-part} developed an estimator of Kullback-Leibler (KL) divergence for one dimensional
sample space based on partitioning the sample space into sub-intervals. The partition depends on the observed sample points. 
They show that this estimator converges to the true KL divergence almost surely as the number of sub-intervals and the number of 
sample points inside each sub-interval go to infinity. However, they do not provide a convergence rate result for this method.    

In this paper, we investigate a multivariate generalization of the data-dependent partition scheme proposed in  
\citep{wang2005_diverg-est-cont-distr-data-dept-part} to estimate the $\phi$-divergence. 
We also generalize the analysis for a family of $\phi$-divergence. 
The generalized method is based on partitioning the sample space 
into finitely many hyper-rectangles, using counts of samples in each hyper-rectangle to estimate the divergence 
within in the hyper-rectangle, and finally summing up the estimated divergence in every hyper-rectangle.
We provide a convergence rate result for this method (Theorem~\ref{thm:concent-phi-div-estor}). 
The results are proved under the assumption that the probability density functions satisfy the 
power law and certain additional regularity assumptions,
which are satisfied by most well known continuous distributions 
such as Gaussian, exponential, $\chi^2$, etc., and all light-tailed distributions.


\section{A Data Dependent Partition Method for $\phi$-Divergence Estimation}
\label{sec:part_method}
In this section, we describe a discretization method 
to estimate the $\phi$-divergence $D_{\phi}(P||Q)$ for two unknown probability measures
$P$ and $Q$ defined on $(\mathbb{R}^d, \mathcal{B}_{\mathbb{R}^d})$, where $\mathcal{B}_{\mathbb{R}^d}$ is the Borel measure
on $\mathbb{R}^d$. We assume that $P$ is absolutely continuous with respect to $Q$ (denoted as $P\ll Q$), otherwise the divergence is infinity. 
We also assume the densities of $P$ and $Q$ with respect to Lebesgue measure exist, which are denoted as $p(\bs{x})$ and $q(\bs{x})$, respectively. 
Suppose the random vectors $\bs{X}$ and $\bs{Y}$ follow the distributions of $P$ and $Q$ respectively, and there are i.i.d. observations $\{\bs{X}_i\}^{n_1}_1$ of $\bs{X}$ and $\{\bs{Y}_i\}^{n_2}_1$ of $\bs{Y}$. The sample space $\mathbb{R}^d$ is partitioned into a number of hyperrectangles (or $d$-dimensional rectangles). 
The estimator of divergence is the sum of estimation in each hyperrectangle using counts of empirical data that fall into that hyperrectangle.
In this section, we provide the convergence result to a general $\phi$-divergence. 
In Sections~\ref{sec:concent-int-error} and \ref{sec:concent-div-estor}, 
we provide a rate of convergence on the number of samples needed to ensure a probability guarantee of this method. 
Before introducing the space partitioning algorithm, we give the following definitions. 
The space partitioning algorithm is described as Algorithm~\ref{alg:partition}.  
Our analysis is based on density functions $p(\x)$ and $q(\x)$ satisfying the power law regularity condition
defined as follows:   
\begin{definition}
A probability density function $f$ satisfies the \emph{power law regularity condition with parameters} ($c>0,\alpha>0)$ if 
\begin{equation}
\int_{\{\bs{x}\in\mathbb{R}^d:\; \|x\|>r\}} f(\bs{x})d\bs{x} < \frac{c}{r^{\alpha}},
\end{equation} 
for any $r>0$.
\end{definition}

\begin{definition}
Let $\phi$ be an univariate differentiable function defined on $[0,\infty)$. The function $\phi$ is $\varepsilon$-regularized by 
$\{ K_0(\cdot), K_1(\cdot,\cdot), K_2(\cdot) \}$ if for any $\varepsilon>0$ small enough and $L>\varepsilon$, the following regularity conditions hold:
\begin{align}
\textrm{(a)} &\;\;  \underset{s\in[0, L]}{\textrm{max}} |\phi(s)| \le K_0(L)  \nonumber\\
\textrm{(b)} &\;\;  \underset{s\in[\varepsilon, L]}{\textrm{max}}\; |\phi^{\prime}(s)| \le K_1(\varepsilon,L) \nonumber\\
\textrm{(c)} &\;\;  |\phi(s_2) - \phi(s_1)| \le K_2(\varepsilon) \quad \forall s_1,s_2\in[0,\varepsilon]. \nonumber
\end{align}
For any $\varepsilon, L>0$, let $K(\varepsilon,L):=\textrm{max}\{ K_0(L), K_1(\varepsilon,L), K_2(\varepsilon) \}$.
\end{definition}
\begin{remark}
If a function $\phi$ is $\varepsilon$-regularized by $\{ K_0(\cdot), K_1(\cdot,\cdot), K_2(\cdot) \}$, then the following inequality holds:
\begin{equation}
|\phi(s_2)-\phi(s_1)|\le K_1(\varepsilon,L)|s_2-s_1| + 2K_2(\varepsilon) \quad \forall s_1\in[0,\varepsilon],\; \forall s_2\in[\varepsilon,L] \nonumber
\end{equation}
\end{remark}
Table~\ref{tab:list-reg-func} provides the value of $\{ K_0(\cdot), K_1(\cdot, \cdot), K_2(\cdot) \}$ for some specific choices of $\phi$-divergences. 

\begin{table}[H]
\centering
\caption{List of possible regularization functions for common $\phi$-divergence distances}\label{tab:list-reg-func}
\begin{tabular}{lllll}
\hline\hline 
$\phi$-divergence distance  & $\phi(t)$ & $K_0(L)$ & $K_1(\varepsilon, L)$ & $K_2(\varepsilon)$  \\
\hline 
KL-divergence  &  $t\textrm{log}t$   &  $|L\textrm{log}L|$    &  $\textrm{max}\{ |\textrm{log}\varepsilon|, |\textrm{log}L|  \} + 1$      &   $2|\varepsilon\textrm{log}\varepsilon|$ \\
Hellinger distance & $(\sqrt{t}-1)^2$  & $\textrm{max}\{1, (\sqrt{L} - 1)^2 \}$  & $\textrm{max}\left\{ \left| 1-\frac{1}{\sqrt{\varepsilon}} \right|, \left| 1-\frac{1}{\sqrt{L}} \right| \right\}$  & $2\sqrt{\varepsilon}$ \\
Total variation distance & $\frac{1}{2}|t-1|$  &  $\textrm{max}\left\{\frac{1}{2}, \frac{1}{2}|L-1| \right\}$  & $\frac{1}{2}$ & $\frac{1}{2}\varepsilon$  \\
$\chi^2$-divergence distance & $(t-1)^2$ &  $\textrm{max}\{ 1, (L-1)^2 \}$  &  $\textrm{max}\{2, 2|L-1| \}$  & $2\varepsilon$  \\
\hline\hline
\end{tabular}
\end{table}

\begin{assumption}\label{ass1}
Regularity conditions of the divergence function $\phi$ and probability measures $P$ and $Q$:
\begin{itemize}
	\item[\emph{(a)}] The divergence function $\phi$ is $\varepsilon$-regularized by $\{K_0(\cdot), K_1(\cdot,\cdot), K_2(\cdot) \}$.
	\item[\emph{(b)}] The probability measures $P$ and $Q$ have density functions $p(\bs{x})$ and $q(\bs{x})$, respectively.
	\item[\emph{(c)}] The density functions $p$ and $q$ satisfy the power law regularity condition with parameters $(c,\alpha)$, and they are also Lipschitz continuous with Lipschitz constant $L_1$.
	\item[\emph{(d)}] There exists an $L_2>0$ such that $|p(\bs{x})/q(\bs{x})|<L_2$ for every $\bs{x}\in\mathbb{R}^d$.
\end{itemize}
\end{assumption}

\begin{definition}\label{def:k-level}
An interval $J\subset\mathbb{R}$ is \emph{infinitely large} if $J$ contains $\infty$ or $-\infty$; e.g., $[a,\infty)$ and $\mathbb{R}$ are an infinitely large intervals. 
A set $I=\prod^k_{i=1}J_i\times\mathbb{R}^{d-k}$ is a \emph{$k$-level hyperrectangle} if every $J_i$ ($i\in[k]$) is an interval other than $\mathbb{R}$.
A hyperrectangle $I=\prod^d_{i=1}J_i$ is \emph{infinitely large} if there exists an $i\in[d]$ such that $J_i$ is infinitely large. 
\end{definition}

Now we describe the divergence estimation method in details. Let $P_{n_1}$ and $Q_{n_2}$ be empirical distributions of $P$ and $Q$, respectively.
First, we use Algorithm~\ref{alg:partition} to partition $\mathbb{R}^d$ into $m$ hyperrectangles according to $Q_{n_2}$. 
Denote the partition as $\mathcal{I}=\{I_i\}^m_{1}$.
For $i\in[m]$, let $p_i$ (resp. $q_i$) be the number of samples from $P_{n_1}$ (resp. $Q_{n_2}$) that fall into $I_i$. By the way of partition, $q_i=n_2/m$.
The estimator of $D_{\phi}(P||Q)$ is constructed as
\begin{equation}\label{eqn:phi-diverg-est}
\widehat{D}^{n_1,n_2}_{\phi}(P||Q) = \sum^{m}_{i=1}\phi\left(\frac{P_{n_1}(I_i)}{Q_{n_2}(I_i)}\right), 
\end{equation}
where $P_{n_1}(I_i)=p_i/n_1$ and $Q_{n_2}(I_i)=1/m$.

\begin{algorithm}[H] 
	\caption{An algorithm to partition $\mathbb{R}^d$ into $m$ equal-measure hyperrectangles with respect to an empirical probability measure $P_n$. 
	 }
	\label{alg:partition}
	\begin{algorithmic}
	\State {\bf Input}: An empirical probability measure $\mu_n:=\frac{1}{n}\sum^n_{i=1}\bs{1}_{\{\bs{X}_i\}}$ consisting of $n$ samples $\{\bs{X}_i\}^n_{i=1}$. 
					A partition number $m$ satisfying $m=m^d_0$ for some $m_0\in N^+$. The integers $m$ and $n$ satisfy $m|n$.
	\State {\bf Output}: A partition $\mathcal{I}=\{I_i\}^m_{i=1}$ of $\mathbb{R}^d$, such that $\mu_n(I_i)=1/m$, $\forall i\in[m]$.
	\State {\bf Initialization}: Set $I\gets\mathbb{R}^d$, $k\gets 0$, $r(I)\gets 0$ and $\mathcal{I}\gets\{I\}$.   
		\For{$k=1,\ldots,d$} 
			\For{$I\in\mathcal{I}$ with $v(I)=k-1$}
				\State Suppose $I$ is of the form: $I=\prod^{k-1}_{i=1}J_i\times\mathbb{R}^{d-k+1}$, where every $J_i$ is in one of 
					    the three forms: 
				\State  $(-\infty,a]$, $(a,b]$, or $(b,\infty)$ for some $a$ or $b$. 
				\State  Partition $I$ into $m_0$ subregions 
					    $\{I_j\}^{m_0}_{j=1}$, where 
					    \begin{equation} 
					    I_j=\left\{
					    		\begin{array}{ll}  
					    			\prod^{k-1}_{i=1}J_i\times(-\infty,a_1]\times\mathbb{R}^{d-k} & \textrm{ if } j=1 \\
								\prod^{k-1}_{i=1}J_i\times(a_{j-1},a_j]\times\mathbb{R}^{d-k}  & \textrm{ if } j=2,\ldots,m_0-1 \\
								\prod^{k-1}_{i=1}J_i\times(a_{m_0-1},\infty)\times\mathbb{R}^{d-k}  & \textrm{ if } j=m_0
					    		\end{array}
					    	    \right.
					    \end{equation}
				\State such that $\mu_n(I_j)=1/m^k_0$ for every $j\in[m_0]$.
				\State Let $\mathcal{I}\gets (\mathcal{I}\setminus\{I\})\cup\{ I_j\}^m_1$.
			\EndFor
		\EndFor 	 
		\State Assign index $\{1,\ldots,m\}$ to $m$ hyperrectangles in $\mathcal{I}$ respectively. \Return $\mathcal{I}$.
       	\end{algorithmic}	 
\end{algorithm}

Without loss of generality, let $n_1=n_2=n$.
The error of the estimator $\widehat{D}^{n_1,n_2}_{\phi}(P||Q)$ with respect to the true value $D_{\phi}(P||Q)$
can be bounded by two terms as follows:
\begin{align}\label{eqn:D-D}
& \big|\widehat{D}^{(n,n)}_{\phi}(P||Q) - D_{\phi}(P||Q) \big| 
 = \Bigg| \sum^m_{i=1}\phi\left(\frac{P_n(I_i)}{Q_n(I_i)} \right)Q_n(I_i) -\sum^m_{i=1}\int_{I_i} \phi\left(\frac{dP}{dQ} \right)dQ  \Bigg| \nonumber\\
&\le   \underbrace{\Bigg| \sum^m_{i=1}\phi\left(\frac{P(I_i)}{Q(I_i)} \right)Q(I_i) -  \sum^m_{i=1}\int_{I_i} \phi\left(\frac{dP}{dQ} \right)dQ\Bigg|}_{T_1} \nonumber\\
&\quad + \underbrace{\Bigg| \sum^m_{i=1}\phi\left(\frac{P_n(I_i)}{Q_n(I_i)} \right)Q_n(I_i) -  \sum^m_{i=1}\phi\left(\frac{P(I_i)}{Q(I_i)} \right)Q(I_i) \Bigg|}_{T_2}.
\end{align}
Note that $T_1$ is the error of numerical integration and $T_2$ is the error of random sampling. 
The convergence rate for $T_1$ is given in Section~\ref{sec:concent-int-error},
and the convergence rates for $T_2$ and that for $D_{\phi}(P||Q)$ are given  
in Section~\ref{sec:concent-div-estor}. 

\section{A convergence rate of the data-dependent partition error}
\label{sec:concent_partition}
In this section, we establish a key intermediate result that bounds the error of partition generated from Algorithm~\ref{alg:partition}
in probability.  This result is used to prove the convergence rate presented in 
Sections~\ref{sec:concent-int-error} and \ref{sec:concent-div-estor}. 
Specifically, let $\mu_n$ be an empirical distribution consisting of $n$ samples drawn from $\mu$. 
For a partition $\pi$ of $\mu_n$, the error of $\pi$ is defined as $\sum_{A\in\pi} |\mu_n(A)-\mu(A)|$.
Let $\mathscr{A}$ be a (possibly infinite) family of partitions of $\mathbb{R}^d$. 
The \emph{maximal cell count} of $\mathscr{A}$ is defined as $m(\mathscr{A})=\underset{\pi\in\mathscr{A}}{\textrm{sup}}\; |\pi|.$
The \emph{growth function} is a combinatorial quantity to measure the complexity of a set family \citep{vapnik1971-unif-converg-rel-freq-event-prob}.
One can analogously define the growth function of a partition family $\mathscr{A}$. Specifically,
given $n$ points $x_1,\ldots,x_n\in\mathbb{R}^d$ and let $B=\{x_1,\ldots,x_n\}$. Let $\Delta(\mathscr{A},B)$ be the number of 
distinct partitions of $B$ having the form
\begin{equation}\label{eqn:dist-part}
\{A_1\cap B, \ldots, A_r\cap B \}
\end{equation}
that are induced by partitions $\{A_1,\ldots,A_r\}\in\mathscr{A}$. Note that the order of appearance of the individual sets in
\eqref{eqn:dist-part} is disregarded. The growth function of $\mathscr{A}$ is then defined as
\begin{equation}
\Delta^*_n(\mathscr{A}) = \underset{B\in\mathbb{R}^{n\cdot d}}{\textrm{max}} \; \Delta(\mathscr{A}, B).
\end{equation}
It is the largest number of distinct partitions of any $n$ points subset of $\mathbb{R}^d$ that can be induced by the partitions in $\mathscr{A}$.
The convergence rate of the partition error is given in Lemma~\ref{lem:part-acc-delta}.  
\begin{lemma}[\cite{LN1996_data-driven-hist-methd-density-est-class}]\label{lem:part-acc-orig}
Let  $\mu$ be any $\mathcal{B}_{\mathbb{R}^d}$-measurable probability distribution, and let $\mu_n$ be its empirical distribution with $n$ samples. 
Let $\mathscr{A}$ be any collection of partitions of $\mathbb{R}^d$.  For each $n\ge 1$ and every $\varepsilon>0$,
\begin{equation}
\mathbb{P}\left\{ \underset{\pi\in\mathscr{A}}{\emph{sup}}\;\sum_{A\in\pi} |\mu_n(A)-\mu(A)|>\varepsilon \right\} \le 4\Delta^*_{2n}(\mathscr{A})2^{m(\mathscr{A})}\emph{exp}(-n\varepsilon^2/32).
\end{equation}
\end{lemma}

\begin{lemma}[The error of a partition]\label{lem:part-acc-delta}
Let  $\mu$ be any $\mathcal{B}_{\mathbb{R}^d}$-measurable probability distribution, and let $\mu_n$ be its empirical distribution with $n$ i.i.d. samples $\{\bs{Z}_i\}^n_1$ where $\bs{Z}_i\sim\mu$ for each $i\in[n]$. Let $\mathscr{A}$ be the collection of all possible partitions of $\mathbb{R}^d$ using Algorithm~\ref{alg:partition} for all possible outcomes of $\{\bs{Z}_i\}^{n}_1$, i.e.,
\begin{equation}
\mathscr{A}=\bigcup_{\{\bs{z}_i\}^n_1\in [\emph{supp}(\mu)]^n} \big\{\pi:\; \pi \emph{ is obtained from applying Algorithm~\ref{alg:partition} on } \{\bs{z}_i\}^n_1  \big\},
\end{equation} 
where $\emph{supp}(\mu)$ is the support of $\mu$ in $\mathbb{R}^d$. If $n>N^*(m,d,\varepsilon,\delta)$, then
\begin{equation}\label{eqn:part-acc-delta}
\mathbb{P}\left\{ \underset{\pi\in\mathscr{A}}{\emph{sup}}\;\sum_{A\in\pi} |\mu_n(A) - \mu(A)| >\varepsilon \right\}\le \delta,
\end{equation}
where $N^*(m,d,\varepsilon,\delta):= \emph{max}\left\{ \frac{c_1 m^{d+1/d -1}}{\varepsilon^4},\quad  \frac{c_2\left[\emph{log}(1/\delta) + m \right]}{\varepsilon^2} \right\}$,
and $c_1$, $c_2$ are some constants.
\end{lemma}
\begin{proof}
In Algorithm~\ref{alg:partition}, each $k$-level hyperrectangle is partitioned into $m^{1/d}$ $(k+1)$-level hyperrectangles with equal measure with respect 
to $\mu_n$ for every $k\in\{0,1,\ldots,d-1\}$. Consider a modified algorithm named $ModAlg$, in which each $k$-level hyperrectangle is partitioned into $m^{1/d}$ $(k+1)$-level hyperrectangles in an \emph{arbitrary} way for every $k\in\{0,1,\ldots,d-1\}$, and let $\mathscr{A}^{\prime}$ be the collection of all possible final partitions generated from this modified algorithm. Clearly, we have $\mathscr{A}\subset\mathscr{A}^{\prime}$. Now we estimate the growth function $\Delta^*_{2n}(\mathscr{A}^{\prime})$. For any fixed $2n$ points in $\mathbb{R}^d$, the number of distinct partitions of $ModAlg$ at level 0 is ${2n+m^{1/d}}\choose{m^{1/d}}$.
At level $k$, there are $m^{k/d}$ number of $k$-level hyperrectangles needed to be partitioned. 
For a specific $k$-level hyperrectangle $I^{(k)}$ containing $l$ points, there are ${l+m^{1/d}}\choose{m^{1/d}}$
ways to partition these $l$ points at the level $k$.  Therefore,
\begin{equation}
\Delta^*_{2n}(\mathscr{A})\le \Delta^*_{2n}(\mathscr{A}^{\prime}) \le \prod^{d-1}_{k=0} {{2n+m^{1/d}}\choose{m^{1/d}}}^{m^{k/d}}={{2n+m^{1/d}}\choose{m^{1/d}}}^{m^{(d-1)/2}}.
\end{equation}
Now we need to find a threshold of $n$ to ensure 
\begin{equation}\label{eqn:acc-ensure1}
4\Delta^*_{2n}(\mathscr{A})2^{m(\mathscr{A})}\textrm{exp}(-n\varepsilon^2/32) < \delta
\end{equation}
in Lemma~\ref{lem:part-acc-orig}. Note that $m(\mathscr{A})=m$, and take logarithm on both sides of \eqref{eqn:acc-ensure1} we impose that
\begin{equation}\label{eqn:n_ineq1}
\begin{aligned}
&\textrm{log}4 + m\textrm{log}2 + \textrm{log}\Delta^*_{2n}(\mathscr{A})  - \frac{n\varepsilon^2}{32} < \textrm{log}\delta \\
\Longleftarrow\quad& \frac{n\varepsilon^2}{32} - \textrm{log}\Delta^*_{2n}(\mathscr{A}) > \textrm{log}(1/\delta) + m\textrm{log}2 + \textrm{log}4 \\
\Longleftarrow\quad& \frac{n\varepsilon^2}{32} - m^{(d-1)/2}\textrm{log}{{2n+m^{1/d}}\choose{m^{1/d}}} > \textrm{log}(1/\delta) + m\textrm{log}2 + \textrm{log}4.
\end{aligned}
\end{equation}
By the inequality $\textrm{log}{{s}\choose{t}}\le sh(t/s)$, where $h(x)=-x\textrm{log}x - (1-x)\textrm{log}(1-x)$ for $x\in(0,1)$, we have
\begin{equation}\label{eqn:acc-ensure2}
\begin{aligned}
\eqref{eqn:n_ineq1}\Longleftarrow\quad& \frac{n\varepsilon^2}{32} - m^{(d-1)/2}(2n+m^{1/d})h\left( \frac{m^{1/d}}{2n+m^{1/d}} \right) > \textrm{log}(1/\delta) + m\textrm{log}2 + \textrm{log}4.
\end{aligned}
\end{equation}
First, we impose that $n>m^{1/d}$, and then we have 
\begin{align}
&(2n+m^{1/d})h\left( \frac{m^{1/d}}{2n+m^{1/d}} \right) \le 3n h\left(\frac{m^{1/d}}{2n}\right) \nonumber\\
&\le 3n\left[ \frac{m^{1/d}}{2n}\textrm{log}\left(\frac{2n}{m^{1/d}} \right) - \left( 1-\frac{m^{1/d}}{2n} \right)\textrm{log}\left( 1-\frac{m^{1/d}}{2n}\right)  \right] \nonumber\\
&\le 3n\left[ \frac{m^{1/d}}{2n}\textrm{log}\left(\frac{2n}{m^{1/d}} \right) + \frac{m^{1/d}}{2n}  \right],  \label{eqn:acc-ensure3}
\end{align}
where we use the inequality $-(1-x)\textrm{log}(1-x)\le x$ for $x\in(0,1)$ in \eqref{eqn:acc-ensure3}. 
Now substituting \eqref{eqn:acc-ensure3} in \eqref{eqn:acc-ensure2}, we obtain that 
\begin{equation}\label{eqn:acc-ensure4}
\begin{aligned}
\eqref{eqn:acc-ensure1}\Longleftarrow &\left[ \frac{\varepsilon^2}{32} - 3m^{\frac{d-1}{2}}\left(\frac{m^{1/d}}{2n}\right)\textrm{log}\left( \frac{2n}{m^{1/d}}\right)    
-3m^{\frac{d-1}{2}}\left(\frac{m^{1/d}}{2n}\right) \right]\cdot n \\
&> \textrm{log}(1/\delta) + m\textrm{log}2 + \textrm{log}4.
\end{aligned}
\end{equation}
To ensure \eqref{eqn:acc-ensure4}, we further impose that
\begin{equation}\label{eqn:acc-ensure5}
3m^{\frac{d-1}{2}}\left(\frac{m^{1/d}}{2n}\right)\textrm{log}\left( \frac{2n}{m^{1/d}}\right) \le \frac{\varepsilon^2}{96}.
\end{equation}
Let $k=2n/m^{1/d}$, we have
\begin{align}\label{eqn:acc-ensure6}
\eqref{eqn:acc-ensure5}\Longleftarrow\quad \frac{\textrm{log}k}{k} \le \frac{\varepsilon^2}{288m^{\frac{d-1}{2}}}\quad \Longleftarrow\quad  \frac{2}{\sqrt{k}}\le \frac{\varepsilon^2}{288m^{\frac{d-1}{2}}},
\end{align}
where we use the inequality $\textrm{log}k\le 2\sqrt{k}$ in the second inequality of \eqref{eqn:acc-ensure6}. Then \eqref{eqn:acc-ensure6} implies that we should
impose 
\begin{equation}
n \ge \frac{2\cdot 288^2 m^{d+1/d -1}}{\varepsilon^4}
\end{equation} 
to ensure \eqref{eqn:acc-ensure5}. Now substitute \eqref{eqn:acc-ensure5} in \eqref{eqn:acc-ensure4}, we have 
\begin{equation}\label{eqn:acc-ensure7}
\eqref{eqn:acc-ensure1}\Longleftarrow\quad \frac{\varepsilon^2}{96}n > \textrm{log}(1/\delta) + m\textrm{log}2 + \textrm{log}4.
\end{equation}
Combine \eqref{eqn:acc-ensure6} and \eqref{eqn:acc-ensure6}, we obtain that if 
\begin{equation}
n > N^*(m,d,\varepsilon,\delta):= \textrm{max}\left\{ \frac{2\cdot 288^2 m^{d+1/d -1}}{\varepsilon^4},\quad  \frac{96\left[\textrm{log}(1/\delta) + m\textrm{log}2 + \textrm{log}4\right]}{\varepsilon^2} \right\}, 
\end{equation}
then \eqref{eqn:part-acc-delta} holds.
\end{proof}

\section{A convergence rate of the approximation error of numerical integration}
\label{sec:concent-int-error}
In this section, we provide a rate of convergence for $T_1$. Note that the partition algorithm only ensures 
that the empirical measures of all hyperrectangles in the partition are equal, however it does not guarantee that
the size of each hyperrectangle is small.
The idea of bounding $T_1$ is to show that with a desired high probability, most of the hyperrectangles 
in the partition obtained from Algorithm~\ref{alg:partition} are in small size as $n$ goes to infinity. 
Then the error of numerical integral can be estimated based on the Lipschitz condition of density 
functions. 
Denote $H_R$ as the $d$-dimensional hypercube $[-R,R]^d$. Let $\partial H_R$ be the boundary of $H_R$.
For a bounded hyperrectangle $I=\prod^d_{j=1}(a_j, b_j]$, let $v(I):=\prod^d_{j=1}|b_j-a_j|$ and $l(I):=\underset{j\in[d]}{\textrm{max}}\; |b_j-a_j|$
be the volume and length of $I$. For an empirical distribution of $\mu_n$ induced by $n$ i.i.d. samples following the probability measure $\mu$, 
apply Algorithm~\ref{alg:partition} to divide $\mathbb{R}^d$ into $m$ hyperrectangles $\mathcal{I}:=\{I_i\}^m_1$ with respect to $\mu_n$. 
Divide $\mathcal{I}$ into four classes based on the position and size of the hyperrectangle. The definition of the four classes are
\begin{equation}
\nonumber
\begin{aligned}
&\Gamma_1(\mu_n,m,H_R):=\{i\in[m]:\; I_i\cap\partial H_R\neq\emptyset \}, \\ 
&\Gamma_2(\mu_n,m,H_R):=\{i\in[m]:\; I_i\subset\mathbb{R}^d\setminus H_R \}, \\
&\Gamma_3(\mu_n,m,H_R):=\left\{ i\in[m]:\; I_i\subset H_R \textrm{ and } l(I_i)\ge m^{-\frac{2d+1}{2d(d+1)}}  \right\}, \\
&\Gamma^j_3(\mu_n,m,H_R):=\Big\{ i\in[m]:\; I_i\subset H_R, \textrm{ and the length of the } j\textrm{-th}\textrm{ interval of } I_i  \\
&\qquad\qquad\qquad\qquad \textrm{ is no less than } m^{-\frac{2d+1}{2d(d+1)}}  \Big\} \quad \forall j\in[d].
\end{aligned}
\end{equation}
We note that $\Gamma_1(\mu_n,m,H_R)$ is the index set of hyperrectangles that intersect with the boundary of $H_R$.
The $\Gamma_2(\mu_n,m,H_R)$ is the index set of hyperrectangles that fall out of $H_R$. The $\Gamma_3(\mu_n,m,H_R)$
and $\Gamma^j_3(\mu_n,m,H_R)$ are the index sets of hyperrectangles whose maximum edge length and the $j$-th edge
length is no less than $m^{-\frac{2d+1}{2d(d+1)}}$, respectively. 
Proposition~\ref{prop:cap-ineq} shows that the cardinality of $\Gamma_1(\mu_n,m,H_R)$, 
$\Gamma_2(\mu_n,m,H_R)$ and $\Gamma_3(\mu_n,m,H_R)$ is a small fraction of $m$ 
with a desired high probability as $m$ goes to infinity.

\begin{proposition}[Chernoff Bound \citep{2016BL_concentration-ineq}]
\label{prop:chernoff-bd}
Let $X_1,\ldots,X_n$ be random variables such that $0\le X_i\le 1$ for all $i$. Let $X=\sum^n_{i=1}X_i$ and set $\mu=\mathbb{E}(X)$.
Then, for all $\delta>0$:
\begin{equation}
\mathbb{P}\left\{ X\ge (1+\delta)\mu \right\} \le \emph{exp}\left(-\frac{2\delta^2\mu^2}{n} \right). 
\end{equation}
\end{proposition}

\begin{proposition}\label{prop:cap-ineq}
Let $\mu$ be a probability measure defined on $(\mathbb{R}^d,\mathcal{B}_{\mathbb{R}^d})$. 
Let $f$ be the density function of $\mu$ which satisfies the power law regularity condition with parameters $(c,\alpha)$.
Let $\mu_n$ be an empirical distribution of $\mu$ induced by $n$ i.i.d. samples, and apply Algorithm~\ref{alg:partition}
to divide $\mathbb{R}^d$ into $m$ hyperrectangles $\mathcal{I}:=\{I_i\}^m_1$ with respect to $\mu_n$. The following properties hold:
\begin{itemize}
\item[\emph{(a)}] $|\Gamma_1(\mu_n,m,H_R)| \le 2dm^{\frac{d-1}{d}}$ \emph{a.s.}, and $|\Gamma_3(\mu_n,m,H_R)|\le 2dRm^{\frac{2d^2+2d-1}{2d^2+2d}}$ \emph{a.s.} 
\item[\emph{(b)}] For every $\varepsilon, \delta >0$, if $n>\frac{2\emph{log}1/\delta}{\varepsilon^2}$ and $R>\left( \frac{2c}{\varepsilon}\right)^{1/\alpha}$,
               \newline then $\mathbb{P}\left\{ |\Gamma_2(\mu_n,m,H_R)|<\varepsilon m  \right\} > 1-\delta.$     
\end{itemize}
\end{proposition}

\begin{proof}
(a) First, we prove $|\Gamma_1(\mu_n,m,H_R)| \le 2dm^{\frac{d-1}{d}}$ a.s. For simplicity, we assume $m^{1/d}$ is an integer.
Since the hyperplanes that compose $\partial H_R$ are either parallel or perpendicular to the boundary of every $I\in\mathcal{I}$,
it is clearly that the following inequality holds:
\begin{equation}
|\Gamma_1(\mu_n,m,H_R)| \le \big| \{i\in[m]:\; I_i \textrm{ is infinitely large} \} \big|.
\end{equation}
It suffices to show that $2dm^{\frac{d-1}{d}}$ is an upper bound of $\big| \{i\in[m]:\; I_i \textrm{ is infinitely large} \} \big|$.
Note that any $I\in\mathcal{I}$ can be written as $I = \prod^d_{j=1}I_j$,
where $I_j$ is the $j$-th component of $I$ for $j\in[d]$.
Since each dimension of $\mathbb{R}^d$ is divided into $m^{1/d}$ intervals, for a fixed $j\in[d]$ we have:
\begin{equation}\label{eqn:num-inf-fix-j}
\big|\{ I\in\mathcal{I}:\; I_j \textrm{ is infinitely large}\} \big| \le 2 \left(m^{1/d}\right)^{d-1},
\end{equation}
as because in $\eqref{eqn:num-inf-fix-j}$ there are only 2 options to choose the form of $I_j$, i.e., either $(-\infty, a]$ or $(a,\infty)$, and
there are $m^{1/d}$ options to choose intervals for $J_{j^{\prime}}$ ($j^{\prime}\neq j$). Then it follows that
\begin{equation}
\nonumber
\big|\{ I\in\mathcal{I}:\; I \textrm{ is infinitely large}\} \big| \le \sum^d_{j=1} \big|\{ I\in\mathcal{I}:\; I_j \textrm{ is infinitely large}\} \big| \le 2dm^{\frac{d-1}{d}}.
\end{equation} 

Next, we prove $|\Gamma_3(\mu_n,m,H_R)|\le 2dRm^{\frac{2d^2+2d-1}{2d^2+2d}}$ a.s.
Consider an arbitary $(k-1)$-level ($k\in[d]$) hyperrectangle denoted as $I=\prod^{k-1}_{j=1}J_j\times\mathbb{R}^{d-k+1}$ In Algorithm~\ref{alg:partition}.
According to the algorithm, $\mu_n(I)=1/m^{(k-1)/d}$. Then the $k$-th component of $I$ is divided into $m^{1/d}$ intervals in the algorithm.
Denote the corresponding $k$-level hyperrectangles to be $\mathcal{I}^{\prime}:=\{ I^{\prime}_i\}^{m^{1/d}}_{1}$, where the first $k-1$ components
of each $I^{\prime}_i$  are the same as that of $I$.  Then $|\Gamma^j_3|\le 2Rm^{\frac{2d+1}{2d(d+1)}}\times (m^{1/d})^{d-1}=2Rm^{\frac{2d^2+2d-1}{2d^2+2d}}$.
This is because the number of intervals within $[-R,R]$ whose length is greater than $m^{-\frac{2d+1}{2d(d+1)}}$ in the $j$-th dimension is less than
$2Rm^{\frac{2d+1}{2d(d+1)}}$, and for any other dimension one can have at most $m^{1/d}$ options to choose the interval. Therefore,
\begin{equation}
|\Gamma_3|\le \sum^d_{j=1}|\Gamma^j_3|\le 2dRm^{\frac{2d^2+2d-1}{2d^2+2d}}. \nonumber
\end{equation}

Now we prove (b). 
Define indicator variables $X_i$ for every $i\in[n]$, such that $X_i=1$ if the $i$-th sample point falls into $\mathbb{R}^d\setminus B(\bs{0},R)$,
otherwise $X_i=0$. Let $X:=\sum^n_{i=1}X_i$.
Since the density function $f$ satisfies the power law regularity condition with parameters $(c,\alpha)$, then:
\begin{equation}
\lambda:=\mathbb{E}[X]=\sum^n_{i=1}\mathbb{E}[X_i]<\frac{cn}{R^{\alpha}}.   \label{eqn:lambda<}
\end{equation}
Since $B(0,R)\subset S$, it follows that
\begin{align}\label{eqn:P<delta}
&\mathbb{P}\left\{ \big|\{ i\in[m]:\; I_i\subset \mathbb{R}^d\setminus S \}\big| \ge \varepsilon m \right\} \le \mathbb{P}\left\{ \big|\{ i\in[m]:\; I_i\subset \mathbb{R}^d\setminus B(\bs{0},R) \}\big| \ge \varepsilon m \right\}  \nonumber\\
&= \mathbb{P}\left\{ X \ge \varepsilon n \right\} =\mathbb{P}\left\{ X\ge \left[ 1 + \left(\frac{\varepsilon n}{\lambda} -1 \right) \right] \lambda \right\} \le \textrm{exp}\left(-\frac{2(\varepsilon n -\lambda)^2}{n} \right) \nonumber\\
&\le \textrm{exp}\left(-\frac{2\left(\varepsilon n - \frac{cn}{R^{\alpha}} \right)^2}{n} \right)\le \textrm{exp}\left( -\frac{n\varepsilon^2}{2}\right)<\delta,
\end{align}
where we use $\lambda<cn/R^{\alpha}$, $R>(2c/\varepsilon)^{1/\alpha}$ and $n>2(\textrm{log}1/\delta)/\varepsilon^2$  in \eqref{eqn:P<delta}. 
\end{proof}

\begin{proposition}\label{prop:int_sum_dy}
For any $\x\in\mathbb{R}^d$, let $x_i$ be the $i$-th component of $\x$. Let $I=\prod^d_{j}(a_j,b_j]$ be any bounded hyperrectangle, then
\begin{equation}
\int_{I}\sum^d_{j=1}|y_j-x_j|d\bs{y} \le \frac{d}{2}l(I)v(I)
\end{equation}
\end{proposition}
\begin{proof}
The integral of interest can be evaluated as follows:
\begin{align}
&\int_{I}\sum^d_{j=1}|y_j-x_j|d\bs{y} = \sum^d_{j=1}\left(\int^{b_j}_{a_j} |y_j-x_j|dy_j\right)\prod_{\{j^{\prime}:\; j^{\prime}\neq j\}}(b_{j^{\prime}}-a_{j^{\prime}}) \nonumber\\
&\le \sum^d_{j=1}\frac{1}{2}|b_j-a_j|^2\prod_{\{j^{\prime}:\; j^{\prime}\neq j\}}(b_{j^{\prime}}-a_{j^{\prime}}) \le \frac{d}{2}l(I)v(I). \nonumber
\end{align}
\end{proof}

\begin{lemma}[The error of numerical integration]
\label{lem:error-num-integ}
Let $P$ and $Q$ be probability measures defined on $(\mathbb{R}^d,\mathcal{B}_{\mathbb{R}^d})$, and $P\ll Q$. 
Let Assumption~\ref{ass1} hold.
Let $P_n$ and $Q_n$ be empirical measures of $P$ and Q with $n$ samples for each probability measures, respectively.
Divide $\mathbb{R}^d$ into $m$ equal-measured hyperrectangles $\{I_i\}^m_i$ with respect to $Q_n$. 
If $n=O\left( m^{d+1/d+3} + m^2\emph{log}1/\delta  \right)$, 
and $m=C(c,\alpha,d,L_1,L_2)\cdot \left[K_1\left(\varepsilon_1, L_2 \right)\right]^{\frac{1+\alpha}{\alpha}\cdot(2d^2+2d)} \cdot  \varepsilon^{-\textrm{max}\left\{ \frac{ 2d^2+2d }{\alpha}, 2d \right\} }$, where $\varepsilon_1$ is the largest value such that 
$K_2(\varepsilon_1)\le \frac{\varepsilon}{10}$, and  $C(c,\alpha,d,L_1,L_2)$ is a constant only depending on $c,\alpha,d,L_1,L_2$,
then
\begin{equation}\label{eqn:int_error}
\mathbb{P}\left\{ \Bigg| \sum^m_{i=1}\int_{I_i}\phi\left(\frac{p(\bs{x})}{q(\bs{x})}\right)q(\bs{x})d\bs{x} - \sum^m_{i=1}\phi\left(\frac{P(I_i)}{Q(I_i)}\right)Q(I_i) \Bigg|>\varepsilon \right\} < \delta.
\end{equation} 
\end{lemma}
To prove this lemma, we introduce an intermediate error notation $\varepsilon_2>0$ for the sake of clarity, 
we will eventually rewrite $\varepsilon_2$ using $\varepsilon_1$ and $\varepsilon$.
We set $R>\left(\frac{2c}{\varepsilon_2} \right)^{1/\alpha}$,
and use notations $\Gamma_1$, $\Gamma_2$ and $\Gamma_3$ to briefly represent $\Gamma_1(Q_n,m,H_R)$, 
$\Gamma_2(Q_n,m,H_R)$ and $\Gamma_3(Q_n,m,H_R)$, respectively. 
Without loss of generality, we assume that $Q(I_i)>0$ for each $i\in[m]$. Otherwise, suppose $Q(I_i)=0$ for some $i\in[m]$. 
It implies that $P(I_i)=0$ as $P\ll Q$, and hence $p(\bs{x})=0$ and $q(\bs{x})=0$ for every $\bs{x}\in I_i$. Then we can 
remove the corresponding term of $I_i$ from the left side of \eqref{eqn:int_error}. 
The proof of Lemma~\ref{lem:error-num-integ} can be decomposed by proving the following intermediate inequalities:
\begin{itemize}
\item[(a)] $\displaystyle T_1\le K_1(\varepsilon_1,L_2) \sum_{i\in[m]}\int_{I_i} \bigg|\frac{p(\bs{x})}{q(\bs{x})} -  \frac{P(I_i)}{Q(I_i)} \bigg| q(\bs{x})d\bs{x} + 2K_2(\varepsilon_1)$;
\item[(b)] $\displaystyle \sum_{i\in[m]\setminus(\Gamma_1\cup\Gamma_2\cup\Gamma_3)} \int_{I_i}\bigg|\frac{p(\bs{x})}{q(\bs{x})} -  \frac{P(I_i)}{Q(I_i)} \bigg| q(\bs{x})d\bs{x}
			\le  \frac{1}{2}dL_1(L_2+1) m^{-\frac{1}{2d}}$;
\item[(c)] If $n>\textrm{max}\left\{ N^*\left(m, d, \frac{1}{2m}, \frac{1}{3}\delta \right), \; \frac{2\textrm{log}(3/\delta)}{\varepsilon^2_2}  \right\}$,
		 then \newline  $\displaystyle \mathbb{P}\left\{\sum_{i\in\Gamma_1\cup\Gamma_2}\int_{I_i}\bigg| \frac{p(\bs{x})}{q(\bs{x})} - \frac{P(I_i)}{Q(I_i)} \bigg|q(\bs{x})d\bs{x}
				\le \frac{6L_2d}{m^{1/d}} + 3L_2\varepsilon_2 \right\} \ge 1-\frac{2}{3}\delta$;
\item[(d)] $\displaystyle \mathbb{P}\left\{ \sum_{i\in[m]\setminus\Gamma_1\cup\Gamma_2}\int_{I_i}\bigg| \frac{p(\bs{x})}{q(\bs{x})} - \frac{P(I_i)}{Q(I_i)} \bigg|q(\bs{x})d\bs{x}\le 6dL_2m^{-\frac{1}{2d^2+2s}} + \frac{1}{2}dL_1(L_2+1)m^{-\frac{1}{2d}}   \right\}\ge 1-\frac{1}{3}\delta$. 
\end{itemize}

\begin{proof}[Proof of \emph{(a)}]
We can first bound the integration error by summing up the error in each hyperrectangle as follows:
\begin{align}\label{eqn:K*sum_int}
&\Bigg| \sum^m_{i=1}\int_{I_i}\phi\left(\frac{p(\bs{x})}{q(\bs{x})}\right)q(\bs{x})d\bs{x} - \sum^m_{i=1}\phi\left(\frac{P(I_i)}{Q(I_i)}\right)Q(I_i) \Bigg| \nonumber\\
&=\Bigg| \sum^m_{i=1}\int_{I_i}\phi\left(\frac{p(\bs{x})}{q(\bs{x})}\right)q(\bs{x})d\bs{x} - \sum^m_{i=1}\int_{I_i}\phi\left(\frac{P(I_i)}{Q(I_i)}\right)q(\bs{x})d\bs{x} \Bigg| \nonumber\\
&\le \sum^m_{i=1}\int_{I_i} \bigg|\phi\left(\frac{p(\bs{x})}{q(\bs{x})}\right) -  \phi\left(\frac{P(I_i)}{Q(I_i)}\right) \bigg| q(\bs{x})d\bs{x}. 
\end{align}
Let $\Gamma^{\varepsilon_1}:=\left\{ i\in[m]:\; \frac{P(I_i)}{Q(I_i)}\le \varepsilon_1 \right\}$, and let 
$I^{\varepsilon_1}_i:=\left\{ \x\in I_i:\; \frac{p(\x)}{q(\x)}\le \varepsilon_1 \right\}$.  
For clarity of notation, let $W_i(\bs{x}):= \bigg|\phi\left(\frac{p(\bs{x})}{q(\bs{x})}\right) -  \phi\left(\frac{P(I_i)}{Q(I_i)}\right) \bigg|$.   
Then the summation of errors on the right side of \eqref{eqn:K*sum_int} can be partitioned as follows:
\begin{align}\label{eqn:Kint+2K}
&\sum^m_{i=1}\int_{I_i} W_i(\bs{x}) q(\bs{x})d\bs{x} = \sum_{i\in\Gamma^{\varepsilon_1}}\int_{I^{\varepsilon_1}_i} W_i(\bs{x}) q(\bs{x})d\bs{x} 
  + \sum_{i\in\Gamma^{\varepsilon_1}}\int_{I_i\setminus I^{\varepsilon_1}_i} W_i(\bs{x}) q(\bs{x})d\bs{x} \nonumber\\
& \qquad +\sum_{i\in[m]\setminus\Gamma^{\varepsilon_1}}\int_{I^{\varepsilon_1}_i} W_i(\bs{x}) q(\bs{x})d\bs{x}  
  +\sum_{i\in[m]\setminus\Gamma^{\varepsilon_1}}\int_{I_i\setminus I^{\varepsilon_1}_i} W_i(\bs{x}) q(\bs{x})d\bs{x} \nonumber\\
&\le \sum_{i\in\Gamma^{\varepsilon_1}}Q(I^{\varepsilon_1}_i) + \sum_{i\in\Gamma^{\varepsilon_1}}\int_{I\setminus I^{\varepsilon_1}_i}  \left[ K_1(\varepsilon_1,L_2) \bigg|\frac{p(\bs{x})}{q(\bs{x})} -  \frac{P(I_i)}{Q(I_i)} \bigg| + 2K_2(\varepsilon_1) \right] q(\bs{x})d\bs{x} \nonumber \\
& \qquad + \sum_{i\in[m]\setminus\Gamma^{\varepsilon_1}}\int_{I^{\varepsilon_1}_i} \left[ K_1(\varepsilon_1,L_2)\bigg|\frac{p(\bs{x})}{q(\bs{x})} -  \frac{P(I_i)}{Q(I_i)} \bigg| + 2K_2(\varepsilon_1) \right] q(\bs{x})d\bs{x} \nonumber\\
& \qquad + \sum_{i\in[m]\setminus\Gamma^{\varepsilon_1}}\int_{I_i\setminus I^{\varepsilon_1}_i}  K_1(\varepsilon_1,L_2)\bigg|\frac{p(\bs{x})}{q(\bs{x})} - \frac{P(I_i)}{Q(I_i)} \bigg| q(\bs{x})d\bs{x}  \nonumber\\
& \le  K_1(\varepsilon_1,L_2) \sum_{i\in[m]}\int_{I_i} \bigg|\frac{p(\bs{x})}{q(\bs{x})} -  \frac{P(I_i)}{Q(I_i)} \bigg| q(\bs{x})d\bs{x}  \nonumber\\
&\quad  + K_2(\varepsilon_1)\left[ \sum_{i\in\Gamma^{\varepsilon_1}}Q(I^{\varepsilon_1}_i) +\sum_{i\in\Gamma^{\varepsilon_1}}2 Q(I_i\setminus I^{\varepsilon_1}_i) + \sum_{i\in [m]\setminus\Gamma^{\varepsilon_1}}2 Q( I^{\varepsilon_1}_i) \right] \nonumber\\
& \le  K_1(\varepsilon_1,L_2) \sum_{i\in[m]}\int_{I_i} \bigg|\frac{p(\bs{x})}{q(\bs{x})} -  \frac{P(I_i)}{Q(I_i)} \bigg| q(\bs{x})d\bs{x} + 2K_2(\varepsilon_1), 
\end{align}
where we use the fact that $\sum_{i\in\Gamma^{\varepsilon_1}}Q(I^{\varepsilon_1}_i) +\sum_{i\in\Gamma^{\varepsilon_1}}2 Q(I_i\setminus I^{\varepsilon_1}_i) + \sum_{i\in [m]\setminus\Gamma^{\varepsilon_1}}2 Q( I^{\varepsilon_1}_i) 
\le 2\sum_{i\in[m]} Q(I_i) = 2$ in the last inequality.
\end{proof}

\begin{proof}[Proof of \emph{(b)}]
Since we have
\begin{align}
\label{eqn:sum_int_[m]-G1G2G3}
&\sum_{i\in[m]\setminus(\Gamma_1\cup\Gamma_2\cup\Gamma_3)} \int_{I_i} \left|\frac{p(\x)}{q(\x)} - \frac{P(I_i)}{Q(I_i)}  \right| q(\x)d\x \nonumber\\
&=\sum_{i\in[m]\setminus(\Gamma_1\cup\Gamma_2\cup\Gamma_3)} \frac{1}{Q(I_i)}\int_{I_i}\bigg|p(\bs{x})Q(I_i) - P(I_i)q(\bs{x}) \bigg|d\bs{x} \nonumber\\
&=\sum_{i\in[m]\setminus(\Gamma_1\cup\Gamma_2\cup\Gamma_3)} \frac{1}{Q(I_i)} \int_{I_i}\left| \int_{I_i} p(\bs{x})q(\bs{y})d\bs{y} - \int_{I_i} p(\bs{y})q(\bs{x})d\bs{y} \right| d\bs{x} \nonumber\\
&\le \sum_{i\in[m]\setminus(\Gamma_1\cup\Gamma_2\cup\Gamma_3)} \frac{1}{Q(I_i)} \int_{I_i}\bigg| \int_{I_i} p(\bs{x})\big[q(\bs{x}) + L_1\|\bs{y}-\bs{x}\| \big]d\bs{y}  \nonumber\\
&\hspace{2cm} - \int_{I_i} \big[p(\bs{x})-L_1\|\bs{y} - \bs{x} \| \big]q(\bs{x})d\bs{y}   \bigg| d\bs{x}  \nonumber\\
&\le \sum_{i\in[m]\setminus(\Gamma_1\cup\Gamma_2\cup\Gamma_3)} \frac{1}{Q(I_i)} \int_{I_i} \int_{I_i} L_1\|\bs{y}-\bs{x}\|\big[ p(\bs{x}) + q(\bs{x}) \big]d\bs{y} d\bs{x} \nonumber\\
&\le \sum_{i\in[m]\setminus(\Gamma_1\cup\Gamma_2\cup\Gamma_3)} \frac{1}{Q(I_i)} \int_{I_i} \int_{I_i} L_1\sum^d_{j=1}|y_j-x_j| \big[ p(\bs{x}) + q(\bs{x}) \big]d\bs{y} d\bs{x},
\end{align}
where $x_j$ and $y_j$ are $j$-th components of vectors $\bs{x}$ and $\bs{y}$, respectively. 
Applying Proposition~\ref{prop:int_sum_dy}, \eqref{eqn:sum_int_[m]-G1G2G3} can be further upper bounded as:
\begin{align}
&\sum_{i\in[m]\setminus(\Gamma_1\cup\Gamma_2\cup\Gamma_3)} \frac{1}{Q(I_i)}\int_{I_i} \big|p(\bs{x})Q(I_i)-P(I_i)q(\bs{x}) \big|d\bs{x} \nonumber\\
&\le \sum_{i\in[m]\setminus(\Gamma_1\cup\Gamma_2\cup\Gamma_3)}  \frac{1}{Q(I_i)}\int_{I_i} \frac{1}{2}dL_1l(I_i)v(I_i) [p(\bs{x})+q(\bs{x})] d\bs{x} \nonumber\\
&\le \sum_{i\in[m]\setminus(\Gamma_1\cup\Gamma_2\cup\Gamma_3)}  \frac{1}{Q(I_i)}  \frac{1}{2}d L_1(L_2+1)l(I_i)v(I_i)Q(I_i) \nonumber\\
&= \sum_{i\in[m]\setminus(\Gamma_1\cup\Gamma_2\cup\Gamma_3)}  \frac{1}{2}d L_1(L_2+1)l(I_i)v(I_i) \nonumber\\
&\le \sum_{i\in[m]\setminus(\Gamma_1\cup\Gamma_2\cup\Gamma_3)} \frac{1}{2}dL_1(L_2+1) \cdot m^{-\frac{2d+1}{2d^2+2d}} \cdot \left( m^{-\frac{2d+1}{2d^2+2d}}\right)^d \nonumber\\
&=  \sum_{i\in[m]\setminus(\Gamma_1\cup\Gamma_2\cup\Gamma_3)} \frac{1}{2}dL_1(L_2+1) m^{-\frac{2d+1}{2d}} \nonumber\\
&\le  \frac{1}{2}dL_1(L_2+1) m^{-\frac{1}{2d}}. \label{eqn:sum_int_[m]-G1G2G3<dL}
\end{align}
\end{proof}

\begin{proof}[Proof of \emph{(c)}]
By Assumption~\ref{ass1}(d), we have 
\begin{align}\label{eqn:sum-int-P/Q}
\sum_{i\in\Gamma_1\cup\Gamma_2}\int_{I_i}\bigg| \frac{p(\bs{x})}{q(\bs{x})} - \frac{P(I_i)}{Q(I_i)} \bigg|q(\bs{x})d\bs{x} 
\le 2L_2\sum_{i\in\Gamma_1\cup\Gamma_2}Q(I_i).  
\end{align}
Since $n>N^*\left(m, d, \frac{1}{2m}, \frac{1}{3}\delta \right)$ and $Q_n(I_i)=1/m$, by Lemma~\ref{lem:part-acc-delta}, we have
\begin{equation}
\mathbb{P}\left\{ \frac{1}{2m}\le Q(I_i) \le \frac{3}{2m} \right\} = \mathbb{P}\left\{ |Q(I_i) - Q_n(I_i)|\le \frac{1}{2m} \right\} \ge 1- \frac{1}{3}\delta, \quad\textrm{for every }i, \nonumber
\end{equation}
and hence $\displaystyle \mathbb{P}\left\{\sum_{i\in\Gamma_1\cup\Gamma_2} Q(I_i) \le \frac{3}{2m}(|\Gamma_1|+|\Gamma_2|)  \right\} \ge 1- \frac{1}{3}\delta$.
Since $n>\frac{2\textrm{log}(3/\delta)}{\varepsilon^2_2}$ and $R>\left(\frac{2c}{\varepsilon_2} \right)^{1/\alpha}$, by Proposition~\ref{prop:cap-ineq}, 
we have $|\Gamma_1|\le 2dm^{\frac{d-1}{d}}$ and $\mathbb{P}\{ |\Gamma_2|<\varepsilon_2m \} > 1-\frac{1}{3}\delta$, and hence
\begin{equation}\label{eqn:P_sum-Q<1-2delta}
\mathbb{P}\left\{\sum_{i\in\Gamma_1\cup\Gamma_2} Q(I_i) \le \frac{3d}{m^{1/d}} + \frac{3\varepsilon_2}{2}  \right\} \ge 1-\frac{2}{3}\delta.
\end{equation}
The inequalities \eqref{eqn:sum-int-P/Q} and \eqref{eqn:P_sum-Q<1-2delta} imply the inequality of (c).
\end{proof}

\begin{proof}[Proof of \emph{(d)}]
Since for every $i\in[m]\setminus(\Gamma_1\cup\Gamma_2)$,
$I_i$ is a bounded hyperrectangle, we have 
\begin{align}\label{eqn:sum_[m]-G1G2G3}
&\sum_{i\in[m]\setminus(\Gamma_1\cup\Gamma_2)} \int_{I_i}\bigg| \frac{p(\bs{x})}{q(\bs{x})} - \frac{P(I_i)}{Q(I_i)} \bigg|q(\bs{x})d\bs{x} = \sum_{i\in\Gamma_3} \int_{I_i}\bigg| \frac{p(\bs{x})}{q(\bs{x})} - \frac{P(I_i)}{Q(I_i)} \bigg|q(\bs{x})d\bs{x}  \nonumber\\
&\qquad     + \sum_{i\in[m]\setminus(\Gamma_1\cup\Gamma_2\cup\Gamma_3)} \int_{I_i}\bigg| \frac{p(\bs{x})}{q(\bs{x})} - \frac{P(I_i)}{Q(I_i)} \bigg|q(\bs{x})d\bs{x} \nonumber\\
&\le \sum_{i\in\Gamma_3} 2L_2Q(I_i) + \sum_{i\in[m]\setminus(\Gamma_1\cup\Gamma_2\cup\Gamma_3)} \int_{I_i}\bigg| \frac{p(\bs{x})}{q(\bs{x})} - \frac{P(I_i)}{Q(I_i)} \bigg|q(\bs{x})d\bs{x}  \nonumber\\
&\le 2L_2 |\Gamma_3(Q_n,m,H_R)| \cdot \underset{i\in[m]}{\textrm{max}}\; Q(I_i) \nonumber\\
&\qquad +  \sum_{i\in[m]\setminus(\Gamma_1\cup\Gamma_2\cup\Gamma_3)} \int_{I_i}\frac{1}{Q(I_i)} \bigg|p(\bs{x})Q(I_i) - P(I_i)q(\bs{x}) \bigg|d\bs{x} \nonumber\\
&\le 4L_2dRm^{\frac{2d^2+2d-1}{2d^2+2d}}\cdot \underset{i\in[m]}{\textrm{max}}\; Q(I_i) \nonumber\\
&\qquad +  \sum_{i\in[m]\setminus(\Gamma_1\cup\Gamma_2\cup\Gamma_3)} \int_{I_i}\frac{1}{Q(I_i)} \bigg|p(\bs{x})Q(I_i) - P(I_i)q(\bs{x}) \bigg|d\bs{x} 
\end{align}

Now substitute the inequality of (c) in \eqref{eqn:sum_[m]-G1G2G3}, and note that $\mathbb{P}\left\{\underset{i\in[m]}{\textrm{max}}\;Q(I_i) \le \frac{3}{2m} \right\}\ge 1- \frac{1}{3}\delta$ for $n>N^*\left(m, d, \frac{1}{2m}, \frac{1}{3}\delta \right)$. It follows that
\begin{align}\label{eqn:P_int_[m]-G1G2}
&\mathbb{P}\left\{ \sum_{i\in[m]\setminus(\Gamma_1\cup\Gamma_2)} \int_{I_i}\bigg| \frac{p(\bs{x})}{q(\bs{x})} - \frac{P(I_i)}{Q(I_i)} \bigg|q(\bs{x})d\bs{x} \le 6 dL_2R m^{-\frac{1}{2d^2+2d}} + \frac{1}{2}dL_1(L_2+1) m^{-\frac{1}{2d}}   \right\}\nonumber\\
&\hspace{1cm} \ge 1- \frac{1}{3}\delta.
\end{align} 
\end{proof}

\begin{proof}[Proof of Lemma~\ref{lem:error-num-integ}] 
Now incorporating inequalities from (a), (c) and (d), we obtain
\begin{align}
&\mathbb{P}\bigg\{ \Bigg| \sum^m_{i=1}\int_{I_i}\phi\left(\frac{p(\bs{x})}{q(\bs{x})}\right)q(\bs{x})d\bs{x} - \sum^m_{i=1}\phi\left(\frac{P(I_i)}{Q(I_i)}\right)Q(I_i) \Bigg| \nonumber\\
&\quad < 2K_2(\varepsilon_1) + 2K_1(\varepsilon_1,L_2)L_2\left( \frac{3d}{m^{1/d}} + \frac{3\varepsilon_2}{2}\right) \nonumber\\
&\qquad + 6K_1(\varepsilon_1,L_2)L_2Rd\cdot m^{-\frac{1}{2d^2+2d}} + \frac{1}{2}K_1(\varepsilon_1,L_2)L_1(L_2+1)d\cdot m^{-\frac{1}{2d}}  \bigg\}  \ge 1-3\delta_1. \nonumber
\end{align} 
Now we set $\delta_1=\delta/3$, $2K_2(\varepsilon_1)<\varepsilon/5$, $6K_1(\varepsilon_1,L_2)L_2d m^{-1/d} < \varepsilon/5$, $3K_1(\varepsilon_1,L_2)L_2\varepsilon_2<\varepsilon/5$,  
\newline
$6K_1(\varepsilon_1,L_2)L_2Rdm^{-\frac{1}{2d^2+2d}} < \varepsilon/5$
and $\frac{1}{2}K(\varepsilon_1,L_2)L_1(L_2+1)dm^{-\frac{1}{2d}}<\varepsilon/5$, respectively. These settings ensure that 
\begin{equation}
\mathbb{P}\left\{\Bigg| \sum^m_{i=1}\int_{I_i}\phi\left(\frac{p(\bs{x})}{q(\bs{x})}\right)q(\bs{x})d\bs{x} - \sum^m_{i=1}\phi\left(\frac{P(I_i)}{Q(I_i)}\right)Q(I_i) \Bigg| <\varepsilon \right\} \ge 1-\delta,
\end{equation}
under the conditions that $n=O\left( m^{d+1/d+3} + m^2\textrm{log}1/\delta  \right)$, and $m=C(c,\alpha,d,L_1,L_2)\cdot \left[K_1(\varepsilon_1, L_2)\right]^{\frac{1+\alpha}{\alpha}\cdot(2d^2+2d)} \cdot  \varepsilon^{- \textrm{max}\left\{ \frac{ 2d^2+2d }{\alpha}, 2d \right\} }$, where $\varepsilon_1$ is the largest value such that $K_2(\varepsilon_1)\le \frac{\varepsilon}{10}$. 
\end{proof}

\section{Concentration of the $\phi$-divergence estimator}
\label{sec:concent-div-estor}
\begin{theorem}[Concentration of the $\phi$-divergence estimator]\label{thm:concent-phi-div-estor}
Let $P$ and $Q$ be probability measures defined on $(\mathbb{R}^d,\mathcal{B}_{\mathbb{R}^d})$, and $P\ll Q$. 
Let the Assumption~\ref{ass1} hold. Let $P_n$ and $Q_n$ be empirical measures of $P$ and Q with $n$ samples for each probability measures, respectively.
Divide $\mathbb{R}^d$ into $m$ equal-measured hyperrectangles $\mathcal{I}=\{I_i\}^m_i$ with respect to $Q_n$. If $n=O\bigg( \emph{max}\bigg\{  m^{d+1/d+3}, \; m^2\emph{log}1/\delta, \; [K_{12}(\varepsilon/9, L_2)]^4\cdot \frac{m^{d+1/d-1}}{\varepsilon^4}, \;$ \newline $[K_{12}(\varepsilon/9,L_2)]^2\frac{\emph{log}1/\delta}{\varepsilon^2} \bigg\} \bigg)$, and $m=C(c,\alpha,d,L_1,L_2)\cdot \left[K\left( \frac{\varepsilon}{5K_3}, L_2 \right)\right]^{\frac{1+\alpha}{\alpha}\cdot(2d^2+2d)} \cdot  \varepsilon^{-\emph{max}\left\{ \frac{ 2d^2+2d }{\alpha}, 2d \right\} }$, then
\begin{equation}
\mathbb{P}\left\{ \big|\widehat{D}^{(n,n)}_{\phi}(P||Q) - D_{\phi}(P||Q) \big| \le \varepsilon  \right\}\ge 1-\delta.
\end{equation}
\end{theorem}
\begin{proof}
By \eqref{eqn:D-D}, $\displaystyle \big|\widehat{D}^{(n,n)}_{\phi}(P||Q) - D_{\phi}(P||Q) \big|\le T_1 + T_2$, 
where $T_1$ is the approximation error of integration which can be bounded using Lemma~\ref{lem:error-num-integ}.
To bound $T_2$, we split it into the following two terms:
\begin{align}\label{eqn:T2}
T_2\le \underbrace{\sum^m_{i=1} \Bigg| \phi\left(\frac{P_n(I_i)}{Q_n(I_i)} \right) - \phi\left(\frac{P(I_i)}{Q(I_i)} \right) \Bigg|Q_n(I_i)}_{T_{21}} + \underbrace{\sum^m_{i=1} \Bigg|\phi\left(\frac{P(I_i)}{Q(I_i)} \right) \Bigg| |Q_n(I_i) - Q(I_i)|}_{T_{22}},
\end{align}
and then we bound $T_{21}$ and $T_{22}$. In the following analysis, we introduce parameters $\varepsilon_1$, $\varepsilon_2$ and $\delta_1$ to temporarily 
measure errors and uncertainty.
They will be rewritten in terms of $\varepsilon$ and $\delta$ at the end of the proof.
First, $T_{22}$ is more straightforward to estimate:
\begin{align}\label{eqn:T22}
T_{22}\le \left(\underset{s\in[0,L_2]}{\textrm{max}}|\phi(s)|\right)\sum^m_{i=1} |Q_n(I_i) - Q(I_i)| \le K_0(L_2)\varepsilon_1,
\end{align}
with probability greater than $1-\delta_1$ when $n>N^*(m,d,\varepsilon_1,\delta_1)$.  
Next, we bound $T_{21}$. \newline
Let $J^{\varepsilon_2}_1:=\left\{ i\in[m]:\; \frac{P_n(I_i)}{Q_n(I_i)}\le \varepsilon_2 \textrm{ and } \frac{P(I_i)}{Q(I_i)}\le \varepsilon_2  \right\}$, 
\newline $J^{\varepsilon_2}_2:=\left\{ i\in[m]:\; \frac{P_n(I_i)}{Q_n(I_i)}\le \varepsilon_2 \textrm{ and } \frac{P(I_i)}{Q(I_i)}> \varepsilon_2  \right\}\cup 
\left\{ i\in[m]:\; \frac{P_n(I_i)}{Q_n(I_i)}> \varepsilon_2 \textrm{ and } \frac{P(I_i)}{Q(I_i)}\le \varepsilon_2  \right\}$, 
\newline $J^{\varepsilon_2}_3:=\left\{ i\in[m]:\; \frac{P_n(I_i)}{Q_n(I_i)}> \varepsilon_2 \textrm{ and } \frac{P(I_i)}{Q(I_i)}> \varepsilon_2  \right\}$,
and $W_i= \Big| \phi\left(\frac{P_n(I_i)}{Q_n(I_i)} \right) - \phi\left(\frac{P(I_i)}{Q(I_i)} \right) \Big|$.
Then
\begin{align}\label{eqn:T21}
&T_{21} = \sum_{i\in J^{\varepsilon_2}_1} W_i Q_n(I_i) + \sum_{i\in J^{\varepsilon_2}_2} W_i Q_n(I_i) + \sum_{i\in J^{\varepsilon_2}_3} W_i Q_n(I_i) \nonumber\\
&\le K_2(\varepsilon_2)\sum_{i\in J^{\varepsilon_2}_1} Q_n(I_i) + \sum_{i\in J^{\varepsilon_2}_2} \left[ K_1(\varepsilon_2,L_2) \left|\frac{P_n(I_i)}{Q_n(I_i)} - \frac{P(I_i)}{Q(I_i)} \right| + 2K_2(\varepsilon_2) \right]Q_n(I_i) \nonumber\\
&\qquad + \sum_{i\in J^{\varepsilon_2}_3} K_1(\varepsilon_2, L_2) \left|\frac{P_n(I_i)}{Q_n(I_i)} - \frac{P(I_i)}{Q(I_i)} \right|Q_n(I_i)  \nonumber\\
&\le 3K_2(\varepsilon_2) + K_1(\varepsilon_2,L_2)\cdot \sum^m_{i=1}  \left|\frac{P_n(I_i)}{Q_n(I_i)} - \frac{P(I_i)}{Q(I_i)} \right|Q_n(I_i)  
\end{align}
Note that the term $ \sum^m_{i=1} \Big| \frac{P_n(I_i)}{Q_n(I_i)} - \frac{P(I_i)}{Q(I_i)}  \Big| Q_n(I_i)$ in the above inequality can be bounded as
\begin{align}\label{eqn:int_in_T21}
& \sum^m_{i=1} \Bigg| \frac{P_n(I_i)}{Q_n(I_i)} - \frac{P(I_i)}{Q(I_i)}  \Bigg| Q_n(I_i) = \sum^m_{i=1}\Bigg|P_n(I_i) - P(I_i) + P(I_i) - \frac{Q_n(I_i)}{Q(I_i)}P(I_i)  \Bigg| \nonumber\\
&\le \sum^m_{i=1}|P_n(I_i) - P(I_i)| + \sum^m_{i=1}\frac{P(I_i)}{Q(I_i)}|Q(I_i)-Q_n(I_i)| \nonumber\\
&\le  \sum^m_{i=1}|P_n(I_i) - P(I_i)| + L_2\sum^m_{i=1}|Q(I_i)-Q_n(I_i)| \nonumber\\
&\le (L_2+1)\varepsilon_1
\end{align}
with probability greater than $1-2\delta_1$ when $n>N^*(m,d,\varepsilon_1,\delta_1)$. Incorporate \eqref{eqn:T2}, \eqref{eqn:T22}, \eqref{eqn:T21} and \eqref{eqn:int_in_T21},
we obtain that
\begin{equation}
\mathbb{P}\big\{ T_1\le 3 K_2(\varepsilon_2) + \left[ K_0(L_2) + (L_2+1)K_1(\varepsilon_2,L_2) \right]\varepsilon_1 \big\} \ge 1-2\delta_1.
\end{equation}
Now let $3K_2(\varepsilon_2)\le \varepsilon/3$, $\left[ K_0(L_2) + (L_2+1)K_1(\varepsilon_2,L_2) \right]\varepsilon_1\le \varepsilon/3$ and $\delta_1=\delta/3$,
we have $\mathbb{P}\left\{ T_1\le \frac{2\varepsilon}{3} \right\}\ge 1-\frac{2\delta}{3}$, if $n\ge C(L_2)\cdot \textrm{max}\Big\{ \big[K_1\left( K^{-1}_2(\varepsilon/9), L_2 \right)\big]^4 \cdot\frac{m^{d+1/d-1}}{\varepsilon^4},$ \newline  $\big[K_1\left( K^{-1}_2(\varepsilon/9), L_2 \right)\big]^2\cdot \frac{\textrm{log}1/\delta}{\varepsilon^2} \Big\}$, where $C(L_2)$ is a constant only depends on $K_0(L_2)$ and $L_2$. Now using Lemma~\ref{lem:error-num-integ} to guarantee that $\mathbb{P}\{T_2\le \frac{\varepsilon}{3}\}\ge 1-\frac{\delta}{3}$, we 
eventually obtain $\mathbb{P}\{T_1+T_2\le \varepsilon\}\ge 1-\delta$ if $n=O\bigg( \textrm{max}\bigg\{  m^{d+1/d+3}, \; m^2\textrm{log}1/\delta, \;
[K_{12}(\varepsilon/9, L_2)]^4\cdot \frac{m^{d+1/d-1}}{\varepsilon^4}, \;$ \newline $[K_{12}(\varepsilon/9,L_2)]^2\cdot \frac{\textrm{log}1/\delta}{\varepsilon^2} \bigg\} \bigg)$, 
and $m=C(c,\alpha,d,L_1,L_2)\cdot \left[K\left( \frac{\varepsilon}{5K_3}, L_2 \right)\right]^{\frac{1+\alpha}{\alpha}\cdot(2d^2+2d)} \cdot$\newline  $\varepsilon^{-\textrm{max}\left\{ \frac{ 2d^2+2d }{\alpha}, 2d \right\} }$.
\end{proof}

\section{Concluding Remarks}
We provide a convergence rate result for the $\phi$-divergence estimator constructed based on the data dependent partition scheme. 
The estimation error consists of two parts: the error of numerical integration and the error of random sampling.
We use concentration inequalities to bound both types of errors with a given  probability guarantee.

\section{Acknowledgement}
We acknowledge NSF grants CMMI-1362003 that was used to support this research.

\bibliographystyle{apa}
\bibliography{reference-database-07-03-2017}

\end{document}